\documentclass[11pt]{article}

\usepackage{tikz}
\usepackage{subfigure}
\usepackage[english]{babel}
\usepackage{graphicx}

\usepackage[center]{caption2}
\usepackage{amsfonts,amssymb,amsmath,latexsym,amsthm}
\usepackage{multirow}

\def\ex{\mbox{ex}}
\def\sat{\mbox{sat}}
\def\Sat{\mbox{Sat}}

\newtheorem{myDef}{Definition}
\newtheorem{thm}{Theorem}[section]
\newtheorem{cor}[thm]{Corollary}
\newtheorem{lem}[thm]{Lemma}
\newtheorem{prop}{Proposition}

\newtheorem{claim}{Claim}
\newtheorem{fact}{Fact}

\begin{document}

\title{Minimizing the number of edges in $\mathcal{C}_{\ge r}$-saturated graphs
\thanks{The work was supported by NNSF of China (No. 12071453) and  Anhui Initiative in Quantum Information Technologies (AHY150200) and National Key Research and Development Project (SQ2020YFA070080).}
}
\author{Yue Ma$^a$, \quad Xinmin Hou$^{a,b}$, \quad Doudou Hei$^a$ \quad Jun Gao$^a$\\
\small $^{a}$ School of Mathematical Sciences\\
\small University of Science and Technology of China, Hefei, Anhui 230026, PR China.\\
\small $^b$CAS Wu Wen-Tsun Key Laboratory of Mathematics\\
\small University of Science and Technology of China, Hefei, Anhui, 230026, PR China
}

\date{}

\maketitle

\begin{abstract}
Given a family of graphs $\mathcal{F}$, a graph $G$ is said to be $\mathcal{F}$-saturated if $G$ does not contain a copy of $F$ as a subgraph for any $F\in\mathcal{F}$ but the addition of any edge $e\notin E(G)$ creates at least one copy of some $F\in\mathcal{F}$ within $G$.
The minimum size of an $\mathcal{F}$-saturated graph on $n$ vertices are called the saturation number, denoted by $\sat(n, \mathcal{F})$. Let $\mathcal{C}_{\ge r}$ be the family of cycles of length at least $r$.
Ferrara et al. (2012) gave lower and upper bounds of $\sat(n, C_{\ge r})$ and determined the exact values of $\sat(n, C_{\ge r})$ for $3\le r\le 5$.
In this paper, we determine the exact value of $\sat(n,\mathcal{C}_{\ge r})$ for $r=6$ and $28\le \frac{n}2\le r\le n$ and give new upper and lower bounds for the other cases.
\end{abstract}

\section{Introduction}
Let $G=(V,E)$ be a graph. We call $|V|$ the {\it order} of $G$ and $|E|$ the size of it. If $|V|=n$, we call $G$ an $n$-vertex graph.
Given a family $\mathcal{F}$ of graphs, a graph $G$ is said to be {\em $\mathcal{F}$-saturated} if $G$ does not contain a subgraph isomorphic to any member  $F\in\mathcal{F}$ but  $G+e$ contains at least one copy of some $F\in\mathcal{F}$ for any edge $e\notin E(G)$.
The {\it Tur\'{a}n number} $\ex(n,\mathcal{F})$ of $\mathcal{F}$ is the maximum number of edges in an $n$-vertex $\mathcal{F}$-saturated graph.
The minimum number of edges in an $n$-vertex $\mathcal{F}$-saturated graph is called the {\it saturation number}, denoted by $\sat(n,\mathcal{F})$, i.e.
$$\sat(n,\mathcal{F})=\min\{|E(G)| : G\mbox{ is an $n$-vertex }\mathcal{F}\mbox{-saturated graph}\}\mbox{.}$$
We call an $n$-vertex $\mathcal{F}$-saturated graph of size $\sat(n,\mathcal{F})$ a {\it minimum extremal graph} for $\mathcal{F}$ and let $\Sat(n,\mathcal{F})$ be the family of all $n$-vertex minimum extremal graphs for $\mathcal{F}$.

Let $C_r$ denote the cycle of length $r$ and $\mathcal{C}_{\ge r}$ be the family of cycles of length at least $r$. Erd\H{o}s and Gallai (1959) proved the following theorem on Tur\'an number of $\mathcal{C}_{\ge  r}$.
\begin{thm}[The Erd\H{o}s-Gallai Theorem, \cite{EG59}]\label{ex}
Let $n\ge r$,
$$\ex(n,\mathcal{C}_{\ge r})\le\frac{(r-1)(n-1)}{2}.$$
\end{thm}

For a single cycle $C_r$, there are many results for $\ex(n, C_r)$ and $\sat(n, C_r)$ have been known,
we review some of them in the following.
\begin{itemize}
\item (Simonovits~\cite{Simo74}) $\ex(n, C_{2k+1})=\lceil\frac{n^2}4\rceil$ for sufficiently large $n$;
\item (Erd\H{o}s-Bondy-Simonovits~\cite{Bon-Sim74}, The Even Cycle Theorem) $\ex(n, C_{2k})=O(n^{1+\frac1k})$;
\item (Erd\H{o}s, Hajnal, and Moon~\cite{EHM64})  $\sat(n,C_3)=n-1$ for $n\ge 3$;
\item (Ollmann~\cite{Oll72}, Tuza~\cite{Tuz89}, Fisher et al~\cite{FFL97})  $\sat(n, C_4)=\lfloor\frac{3n-5}{2}\rfloor$ for $n\ge 5$;

\item (Chen~\cite{Che09,Che11}) $\sat(n, C_5)=\lceil\frac{10}{7}(n-1)\rceil$ for $n\ge 21$;

\item(Barefoot et al.~\cite{BCE96} and Zhang et al.~\cite{Zhang15})  $\sat(n,C_6)\le\lfloor\frac{3n-3}{2}\rfloor$ for $n\ge 9$;

\item (F\"{u}redi and Kim~\cite{FK13})  $(1+\frac{1}{r+2})n-1<\sat(n,C_r)<(1+\frac{1}{r-4})n+\binom{r-4}{2}$ for all $r\ge 7$ and $n\ge 2r-5$;

\item
(Clark, Entringer, and Shapiro~\cite{Clark83, Clark92}, Lin et al~\cite{LJZY97})  $\sat(n,C_n)=\lceil\frac{3n}{2}\rceil$ for $n=17$ or $n\ge 19$.
\end{itemize}
A natural question is to determine $\sat(n,\mathcal{C}_{\ge r})$ for $n\ge r\ge 3$. It is trivial that $\sat(n,\mathcal{C}_{\ge 3})=n-1$ and $\Sat(n,\mathcal{C}_{\ge 3})=\{\mbox{tree on } n \mbox{ vertices}\}$.
Ferrara et al.~\cite{Subdivision12} proved that
\begin{thm}[Ferrara et al., Theorems 2.1, 2.13 and 2.17 in~\cite{Subdivision12}]\label{THM: subdivision}
(1) For $r\ge 3$ and $n\ge n(r)$, there exists an absolute constant $c$ such that
$$\frac{5n}{4}\le\sat(n, \mathcal{C}_{\ge r})\le\left(\frac 54+\frac cr\right)n.$$
In particular, if $r\ge 36$, $c=8$ will suffice.

(2) For $n\ge 1$, $\sat(n, \mathcal{C}_{\ge4})=n +\lfloor\frac{n-3}4\rfloor$.

(3) For $n\ge 5$, $\sat(n, \mathcal{C}_{\ge5})=\lfloor\frac{10(n-1)}7\rfloor$.

\end{thm}

In this paper, we determine the exact values of $\sat(n,\mathcal{C}_{\ge r})$ for $r=6$ and $r$ with $56\le r\le n\le 2r$, and give new lower and upper bounds of $\sat(n,\mathcal{C}_{\ge r})$.
The main results of the paper are the following.
\begin{thm}\label{THM: 6}




For $n\ge6$,
     \begin{equation*}
     	sat(n,\mathcal{C}_{\ge6})=\begin{cases}
     		9   & n=6;\\
     		11   & n=7;\\
     		12   & n=8;\\
     		13   & n=9;\\
     		\left\lceil \frac{3(n-1)}{2}\right\rceil   & n\ge10.
     	\end{cases}
     \end{equation*}
\end{thm}

\begin{thm}\label{THM: lower}

$\sat(n,\mathcal{C}_{\ge r})\ge n+\frac{r}{2}$ for $2r \ge n\ge r\ge 6$.
\end{thm}
To give the new upper bound of $\sat(n,\mathcal{C}_{\ge r})$ and the exact value of $\sat(n,\mathcal{C}_{\ge r})$ for $\frac n2\le r\le n$, we define a function $g(x)$(see Figure~\ref{fuc}) on $x\in(0,1]\cap\mathbb{Q}$:
\begin{equation}
g(x)=\left\{\begin{array}{ll}
1+\frac{1}{2}x, & \mbox{if } x\in[\frac{1}{2},1],\\
1+\frac{k}{2}x, & \mbox{if } x\in (\frac{1}{2k},\frac{2}{4k-3}],\\
2-\frac{3k-3}{2}x, & \mbox{if } x\in (\frac{2}{4k-3},\frac{1}{2k-2}],

 \end{array} \right.  \mbox{ for $k\ge 2$}.
\end{equation}

\begin{figure}[h]
\centering
\includegraphics[width=5.5in]{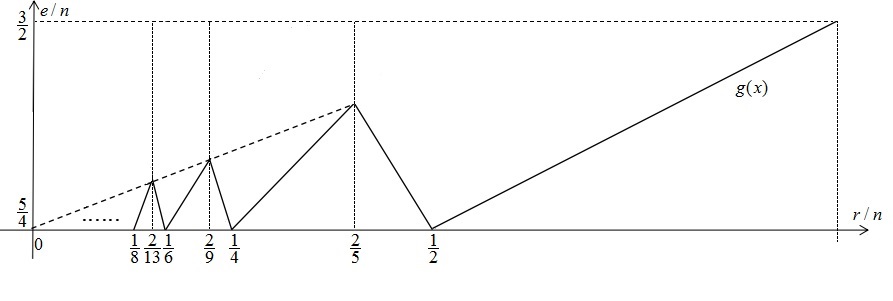}
\caption{The image of $g(x)$}\label{fuc}
\end{figure}

\begin{thm}\label{THM: upper}
(i) $\sat(n,\mathcal{C}_{\ge r})\le g(\frac{r}{n})n+O(\frac{n}{r})$ for $n\ge r\ge 56$.

(ii) $\sat(n,\mathcal{C}_{\ge r})=n+\lceil\frac{r}{2}\rceil$ for  $28\le\frac{n}2\le r\le n$.
\end{thm}

\noindent{\bf Remark:} (1) In fact, we have proved that  $O(\frac nr)<\frac{2n}r$ in the proof of Theorem~\ref{THM: upper}. So, for $n\ge r\ge 56$,
 \begin{equation*}
     	sat(n,\mathcal{C}_{\ge r})\begin{cases}
     		=n+\lfloor\frac r2\rfloor   &  r\le n\le 2r;\\
     		\le (\frac {5k-3}{4k-3}+\frac 2r)n   & 2(k-1)r-2(k-2)\le  n<\frac{4k-3}2r;\\
     		\le (\frac{5k-3}{4k-3}+\frac 2r)n  & \frac{4k-3}2r\le n < 2kr-2(k-1).
     		
     	\end{cases}
 \end{equation*}
The new upper bound is better than the one given in~\ref{THM: subdivision} for the first case and the other case when $n$ is large enough such as $n>r^2$.

(2) Theorem~\ref{THM: upper} (ii) gives the exact value of $\sat(n,\mathcal{C}_{\ge r})$ when $28\le\frac n2\le r\le n$, however, the lower bound given in Theorem~\ref{THM: lower} holds for $2r\ge n\ge r\ge 6$.

(3) Ferrara et al.~\cite{Subdivision12} also observed that for large $n$, $\sat(n,C_r)$ and $\sat(n, \mathcal{C}_{\ge r})$ agree for $r=3$ and 5 and differ for all other values of $r$, save perhaps for $r=6$.
From Barefoot et al.~\cite{BCE96}, Zhang et al.~\cite{Zhang15}) and Theorem~\ref{THM: 6}, we know that $\sat(n,C_6)<\sat(n, \mathcal{C}_{\ge 6})$ when $3n-3$ is odd and  $n\ge 10$.

The rest of the article is arranged as follows. We give the proof of Theorem~\ref{THM: 6} in Section 2. In Section 3, we will give a structural theorem for $\mathcal{C}_{\ge r}$-saturated graphs and the proof of Theorem~\ref{THM: lower}. We prove Theorem~\ref{THM: upper} in Section 4 and give some remarks in the last section.

\section{Proof of Theorem~\ref{THM: 6}}

The following result is due to Dirac.

\begin{thm}[Dirac 1952, Theorem 4 in~\cite{Dirac52}]\label{THM: Dirac52}
Let $G$ be a connected graph with $\delta(G)\ge d$. If $n\ge 2d$ then $G$ contains a path of length at least $2d$.
\end{thm}
Note that a cycle is 2-connected. From the definition of the $\mathcal{C}_{\ge r}$-saturated graph, we have the following two facts.
\begin{fact}\label{FACT: sm}
A $\mathcal{C}_{\ge r}$-saturated graph $G$ on $n$ vertices must be connected.
\end{fact}

\begin{fact}\label{FACT: f2}
Let $G$ be a $\mathcal{C}_{\ge r}$-saturated graph. Then any pair of nonadjacent vertices in $G$ must be connected by a path of length at least $r-1$ in $G$.
\end{fact}

Given integers $n\ge k\ge 2r$, let $H(n,k,r)$ be the graph obtained from the complete graph $K_{k-r}$ by connecting each vertex of the empty graph $\overline{K_{n-k+r}}$ to the same $r$ vertices of $K_{k-r}$, we call the $r$ vertices of $K_{k-r}$ the {\it center} of $H(n,k,r)$.
The following result is due to Kopylov~\cite{Kopylov77}
\begin{thm}[Kopylov~\cite{Kopylov77}]\label{THM:Koplov}
Let $n\ge k\ge 5$ and let $r=\lfloor\frac{k-1}2\rfloor$. If $G$ is a 2-connected $n$-vertex graph with
$e(G)\ge\max\{e(H(n, k, 2)), e(H(n, k,r))\}$,
then either $G$ has a cycle of length at least $k$, or $G=H(n,k,2)$ or $G=H(n,k,r)$.
\end{thm}

Note that when $k=6$, $r=\lfloor\frac{k-1}2\rfloor=2$. So we have the following corollary.
\begin{cor}\label{COR: c_5}
$H(n,6,2)$ is $\mathcal{C}_{\ge 6}$-saturated for $n\ge 6$.
\end{cor}
The following theorem due to Whitney~\cite{Whitney32} characterizes the structure of 2-connected graphs.
Given a graph $H$, we call $P$ an {\em $H$-path} if $P$ is nontrivial and meets $H$ exactly in its ends. We call a path connecting vertices $u$ and $v$ a {\it $(u,v)$-path}.
\begin{thm}[Whitney, 1932]\label{THM: EAR-DECOM}
A graph is 2-connected if and only
if it can be constructed from a cycle by successively adding $H$-paths to graph $H$ already constructed.
\end{thm}

Let $D(a,b)$ be the graph on $a+b+3$ vertices whose vertex set is $\{t_1,t_2,t_3\}\cup A\cup B$ with $|A|=a,|B|=b$ and
$$E(D(a,b))=\{t_1t_2,t_1t_3,t_2t_3\}\cup\{ut_1, ut_2 : u\in A\}\cup \{vt_1, vt_3 : v\in B\}\mbox{.}$$
We call $\{t_1, t_2, t_3\}$ the {\it center} of $D(a,b)$. Clearly, the center vertices have degree $a+b+2$, $a+2$ and $b+2$, respectively. It is easy to check that $D(a,b)$ is $\mathcal{C}_{\geq6}$-saturated if $a,b\ge 2$.
	



	
\begin{lem}\label{61}
Let $G$ be a $2$-connected graph on $n\geq6$ vertices. If $G$ is $\mathcal{C}_{\geq6}$-saturated , then $G$ is isomorphic to $H(n,6,2)$ for $n\ge 6$ or $D(a,b)$ for some $a,b\ge 2$ with $a+b+3=n$.
\end{lem}
\begin{proof}
Since $G$ is $2$-connected, $\delta(G)\ge 2$. If $\delta(G)\ge 3$, then by Theorem~\ref{THM: Dirac52}, $G$ has a cycle of length at least $6$, a contradiction. Hence $\delta(G)=2$.

\begin{claim}\label{CL: 2-vertex}
Every vertex of degree two is contained in a triangle.
\end{claim}
Otherwise, suppose there is a vertex $v$ with $N_{G}(v)=\{u_1,u_2\}$ and $u_1u_2\notin E(G)$. By Fact~\ref{PROP: p2}, there is a $(u_1,u_2)$-path $P$ of length at least $5$ in $G$.
Clearly, $v\notin V(P)$. So $C=P+u_1vu_2$ is a cycle in $G$ of length greater than $6$, a contradiction.
\medskip

If $G$ contains a copy of $H$ isomorphic to $K_4$, we claim that $V(G)\setminus V(H)$ is an isolated set in $G$. If not, let $G'=G-V(H)$ and $e=uv$ is an edge in $G'$, then there are two vertex disjoint paths $P_1, P_2$ connecting $u,v$ and $V(H)$ by the well known Menger Theorem. But any pair of vertices in $H$ is connected by a path $P_3$ of length three. So $P_1+P_3+P_2+e$ is a cycle of length at least six, a contradiction. Therefore, all vertices of $V(G)\setminus V(H)$ must be of degree 2 and have common neighbors in $V(H)$ because $G$ is 2-connected and $\mathcal{C}_{\ge 6}$-free. This implies that $G\cong H(n,6,2)$, as desired.

Now suppose $G$ contains no $K_4$. By Claim~\ref{CL: 2-vertex} and  $\delta(G)=2$, $G$ must contain a triangle  with a vertex of degree $2$ and two vertices of degree at least $3$, say $u$ and $v$. Since $G$ is $2$-connected and $\mathcal{C}_{\ge 6}$-saturated, all $(u,v)$-paths must have length $2$ or $3$. Further by Claim~\ref{CL: 2-vertex} and 2-connectivity of $G$, it is easy to show that $G$ contains $K_4^-$ with a vertex of degree 2 as a subgraph.
\begin{claim}\label{CL: K_4-}
There is such a copy $H$ of $K_4^-$ with the property that $V(G)\setminus V(H)$ is an independent set.
\end{claim}
Let $V(H)=\{v_1, v_2, v_3, v_4\}$ and $H=K_4-\{v_2v_4\}$. Without loss of generality, assume $d_G(v_4)=2$.  Let $S=V(G)\setminus V(H)$. If $G[S]$ is not empty, since $G$ is $2$-connected, there is an $H$-path $P$ of length at least 3 connecting $v_i$ and $v_j$ for some $i,j\in\{1,2,3\}$. If there is a $(v_i,v_j)$-path $P'$ of length 3  in $H$, then $P+P'$ is a cycle of length at least $6$, a contradiction. So $P$ has to be of length three and $\{v_i, v_j\}=\{v_1, v_3\}$. Assume $P=v_1w_1w_2v_3$. Now it can be checked that every pair of vertices in $H\cup P$ is connected by a path of length at least three in $H\cup P$. So there is no $H\cup P$-path of length at length at least three in $G$. Therefore, $S'=V(G)\setminus(V(H\cup P))$ is an independent set in $G$.
If $d_G(w_i)\ge 3$ for $i=1,2$, let $w_i'\in N_G(w_i)\cap S'$. By the Menger Theorem, there are two internal vertex disjoint paths $P_1$ and $P_2$ connecting $w_1'$ and $w_2'$. Since $G[S']$ is empty, $P_i$ ($i=1,2$) must contain edges in $H\cup P$, which implies that $P_i$ ($i=1,2$) has length at least three and so $P_1\cup P_2$ is a cycle of length at least 6, a contradiction. Thus, at least one of $w_1, w_2$, say $w_1$, is of degree two. By Claim~\ref{CL: 2-vertex}, $v_1w_2\in E(G)$. Hence $\{v_1, w_1, w_2, v_3\}$ induces a copy $H'$ of $K_4^-$ with $d_G(w_1)=2$.  If $d_G(v_2)=2$ then $H'$ is a copy of $K_4^-$ as claimed.
Now suppose $d_G(v_2)\ge 3$ and $v_2'\in N_G(v_2)\cap S'$. Then $d_G(v_2')=2$ and $N_G(v_2')\setminus\{v_2\}\subset\{v_1, v_3, w_2\}$. But this is impossible, since, otherwise, we can find a cycle of length at least 6  because there is a path of length at least 4 connecting $v_2$ and any one of $\{v_1, v_3, w_2\}$ in $H\cup P$. The claim is true.

\medskip
Let $V(H)=\{v_1, v_2, v_3, v_4\}$ and $H=K_4-\{v_2v_4\}$ is a copy of $K_4^-$ as claimed in Claim~\ref{CL: K_4-}. Assume $d_G(v_4)=2$.  Let $S=V(G)\setminus V(H)$. Then each vertex of $S$ has degree 2 and has neighbors $v_i, v_j$ for some $i,j\in\{1,2,3\}$. Let $A=\{v : N_G(v)=\{v_1, v_2\}\}$, $B=\{v : N_G(v)=\{v_1, v_3\}\}$ and $C=\{v : N_G(v)=\{v_2, v_3\}\}$. Again since $G$ is $C_{\ge 6}$-free, at least one of $A, B, C$ is empty. Without loss of generality, assume $C=\emptyset$, $|A|=a$ and $|B|=b$, i.e. $G\cong D(a,b)$ for $a\ge 0$, $b\ge 0$. We claim that $a\ge 2$ and $b\ge 2$. Clearly, $v_4\in B$. If $a=0$, then $G\cong H(n, 5, 2)$, which is not $\mathcal{C}_{\ge 6}$-saturated. Without loss of generality, assume $a\ge b$.   If $B=\{v_4\}$, then the longest path connecting $v_2, v_4$ has length at most $4$, a contradiction, too. So we have $a\ge b\ge 2$. We are done.

\end{proof}

Let $B_2(G)$ be the set of blocks of $G$ isomorphic to $K_2$ and $b_2(G)=|B_2(G)|$.

\begin{lem}\label{LEM:cut}
Let $G$ be a $\mathcal{C}_{\ge r}$-saturated graph for $r\ge 4$. Then the following holds.

(a) Every block $B$ of $G$  is $\mathcal{C}_{\ge r}$-saturated. Specifically, each block $B$ with $|V(B)|<r$ is a complete graph.

(b) $B_2(G)$ forms a matching of $G$.

\end{lem}
\begin{proof}
(a) Let $B$ be a block of $G$. Since $B$ is a maximal 2-connected subgraph of $G$, any cycle containing edges of $B$  and any path connecting two nonadjacent vertices in $B$ must be totally contained in $B$. Since $G$ is $\mathcal{C}_{\ge r}$-saturated, $B$ contains no cycle of length at least $r$, and any pair of  nonadjacent vertices in $B$ is connected by a path of length $r-1$ in $B$, i.e. $B$ is $\mathcal{C}_{\ge r}$-saturated too.

Specifically, if $|V(B)|<r$ then the longest path in $B$ has length no more than $r-1$. Hence $B$ contains no nonadjacent vertices, i.e., $B$ is a complete graph.

(b) Suppose there is a vertex $u$ incident with two blocks of $B_2(G)$, say $uv_1,uv_2$. Then $v_1v_2\notin E(G)$, otherwise $uv_1,uv_2$ is contained in the triangle $uv_1v_2u$, a contradiction to the fact that $uv_1, uv_2\in B_2(G)$.
So
there exists a $(v_1,v_2)$-path $P$ of length at least $r-1$  in $G$. However, both $uv_1$ and $uv_2$ are cut edges,  which forces that $uv_1,uv_2\in E(P)$, i.e., $P=v_1uv_2$, a contradiction to $|V(P)|=r\ge 4$.

\end{proof}

An {\it$(a,b,c,d,f)$-cactus}, denoted by $T(a,b,c,d,f)$, is a connected graph whose blocks consist of $a$ copies of $K_3$, $b$ copies of $K_4$, $c$ copies of $K_5$, $d$ members in $\{D(r,s) : r,s\ge 2\}$ and $f$ members in $\{H(t,6,2) : t\ge 6\}$.
\begin{lem}\label{62}
A  graph $G$ is $\mathcal{C}_{\ge6}$-saturated if and only if

(i) $G$ is connected and $B_2(G)$ forms a matching of $G$;

(ii) $G$ contains no $T(a, 0,0,0,0)$ with $a\ge 2$;

(iii) the center vertices of $D(r,s)$, $H(t,6,2)$ and the vertices of blocks $K_3, K_4$ can not incident with a cut edge;

(iv)  each component of $G-B_2(G)$ is isomorphic to $K_1$ or $T(a,b,c,d,f)$ for some $c+d+f\ge 1$ if $a+b\le 1$.
\end{lem}
\begin{proof}
{\bf Necessity:} (i) The connectivity of $G$ comes from the definition of $C_{\ge  r}$-saturation of $G$ and by Lemma~\ref{LEM:cut} (b), $B_2(G)$ forms a matching.

(ii) Otherwise, there are two triangles $B_1$ and $B_2$ such that $|V(B_1)\cap V(B_2)|=1$. Suppose $V(B_1)=\{v_1, v_2, x\}$ and $V(B_2)=\{u_1,u_2, x\}$. Then  $G+v_1u_1$ contains a cycle of length at least $6$. However, $v_1u_1$ is in a block of size $5$ in $G+v_1u_1$, which is a contradiction.

(iii) Suppose to the contrary that there is a cut edge $xy$ with $x$ be a center vertex of $D(r,s)$, $H(t,6,2)$ or a vertex of a block isomorphic to $K_3, K_4$. Choose $z$ to be a center vertex other than $x$ in $D(a,b)$ (if $x$ is of degree $r+2$ or $s+2$ then choose $z$ be the center vertex of degree $r+s+2$), $H(t,6,2)$ or a vertex other than $x$ of $K_3, K_4$. Note that $y$ is a cut vertex not contained in a same block with $z$. So $yz\notin E(G)$ and $G+yz$ should contain a cycle  of length at least $6$. However, the edge $yz$ is in a block of $G+yz$ isomorphic to $D(r+1,s), D(r,s+1), F(t+1, 6,2)$ or a block of size at most $5$, which contains no cycle of length at least $6$ by Lemma~\ref{61}, a contradiction.

(iv) We first show that every block of $G$ is isomorphic to one of $\{K_t : 1\le t\le5\}\cup\{D(r,s): r,s\ge 2\}\cup\{H(t,6,2) : t\ge 6\}$. Let $B$ be a block of $G$. If $|V(B)|\le 5$, then $B\cong K_t$ with $1\le t\le 5$ by Lemma~\ref{LEM:cut} (a) and we are done. Now suppose $|V(B)|\ge 6$. Then $B$ is a $2$-connected $\mathcal{C}_{\ge6}$-saturated graph on at least $6$ vertices. By Lemma~\ref{61}, either $B\cong D(r,s)$ with $r,s\ge 2$ or $B\cong H(t,6,2)$ with $t\ge 6$. We are done, too. If a component of $G-B_2(G)$ is $T(a,b, 0,0,0)$ then $a+b\ge 1$. If $a+b=1$, then, by (iii), $G$ must be $T(a,b,0,0,0)$, which  is of order at most 4, a contradiction.

\medskip
\noindent {\bf Sufficient:} By Lemma~\ref{61}, $G$ is $C_{\ge 6}$-free. It is sufficient to show that the addition of any non-edge to $G$ induces cycles of length at least 6. It can be checked that the component $T(a,b,c,d,f)$ of $G-B_2(G)$ is $\mathcal{C}_{\ge 6}$-saturated for $c+d+f\ge 1$ if $a+b\le 1$, and the addition of any non-edge between two components also gives a cycle of length at least $6$ by (iii).
\end{proof}

The following lemma gives the lower bound of $\sat(n, C_{\ge 6})$.
\begin{lem}\label{63}
For $n\ge6$,
     \begin{equation*}
     	sat(n,\mathcal{C}_{\ge6})\begin{cases}
     		=9   & n=6;\\
     		=11   & n=7;\\
     		=12   & n=8;\\
     		=13   & n=9;\\
     		\ge\left\lceil \frac{3(n-1)}{2}\right\rceil   & n\ge10.
     	\end{cases}
     \end{equation*}
\end{lem}
\begin{proof}
Let $G$ be a minimum $\mathcal{C}_{\ge 6}$-saturated graph on $n$ vertices. Let $b_2(G)$, $b_3(G)$, $b_4(G)$, $b_5(G)$, $b(G)$ and $b^*(G)$ denote the number of blocks isomorphic to $K_2$, $K_3$, $K_4$, $K_5$ and members in  $\{D(r,s) : r,s\ge 2\}$ and members in $\{H(t,6,2) : t\ge 6\}$, respectively. Suppose all the blocks of the form $D(r,s)$ are $\{D(r_G^i, s_G^i): i\in[b(G)]\}$ and all the blocks of the form $H(t,6,2)$ are $\{F(t_G^j, 6,2) : j\in[b^*(G)]\}$. In the following, we write $r^i, s^i$ and $t^j$ for $r_G^i, s_G^i$ and $t^j_G$. By Lemma~\ref{62}, each component of $G-B_2(G)$ is isomorphic to $K_1$ or $T(a,b,c,d,f)$. Let $C(G)$ be the set of all components of $G-B_2(G)$. For each component $H\in C(G)$, we have $|V(H)|=1+2a+3b+4c+\sum_{i=1}^d(r_H^i+s_H^i+2)+\sum_{j=1}^{f}(t_H^j-1)$. Since $B_2(G)$ is a matching, the number of components in $G-B_2(G)$ is $b_2(G)+1$. So
$$|V(G)|=\sum_{H\in C(G)}|V(H)|=(b_2+1)+2b_3+3b_4+4b_5+\sum_{i=1}^{b(G)}(r^i+s^i+2)+\sum_{j=1}^{b^*(G)}(t^j-1)\mbox{.}$$
By (i) and (iii) of Lemma~\ref{62},
$$b_2\le 5b_5+\sum_{i=1}^{b(G)}(r^i+s^i)+\sum_{j=1}^{b^*(G)}(t^j-2)\mbox{.}$$
Therefore,
\begin{equation*}
\begin{split}
|E(G)|&=b_2+3b_3+6b_4+10b_5+\sum_{i=1}^{b(G)}(2r^i+2s^i+3)+\sum_{j=1}^{b^*(G)}(2t^j-2)\\
      &=\frac{3}{2}\left(b_2+2b_3+3b_4+4b_5+\sum_{i=1}^{b(G)}(r^i+s^i+2)+\sum_{j=1}^{b^*(G)}(t^j-1)\right)\\
      &+\frac{1}{2}\left(-b_2+3b_4+8b_5+\sum_{i=1}^{b(G)}(r^i+s^i)+\sum_{j=1}^{b^*(G)}(t^j-1)\right)\\
      &\ge \frac{3}{2}(|V(G)|-1)+\frac{1}{2}\left(3b_4+3b_5+b^*(G)\right)\\
      &\ge\frac{3}{2}(n-1).
\end{split}
\end{equation*}
Thus we have $|E(G)|\ge\lceil\frac{3}{2}(n-1)\rceil$ for $n\ge 10$.

For $6\le n\le 9$, since $r^i, s^i\ge 2$ and $t^j\ge 6$, we have $$1+b_2+2b_3+3b_4+4b_5+6b(G)+5b^*(G)\le n\le 9.$$
Hence we can list all of the $C_{\ge 6}$-saturated graphs of order $n=6,7,8,9$ and compare the number of edges of them.
All the minimum $C_{\ge 6}$-saturated graphs of order $n$ with $6\le n\le 11$ are listed in Figure~\ref{sat6}. This completes the proof.
\begin{figure}[h]
\centering
\includegraphics[width=5.5in]{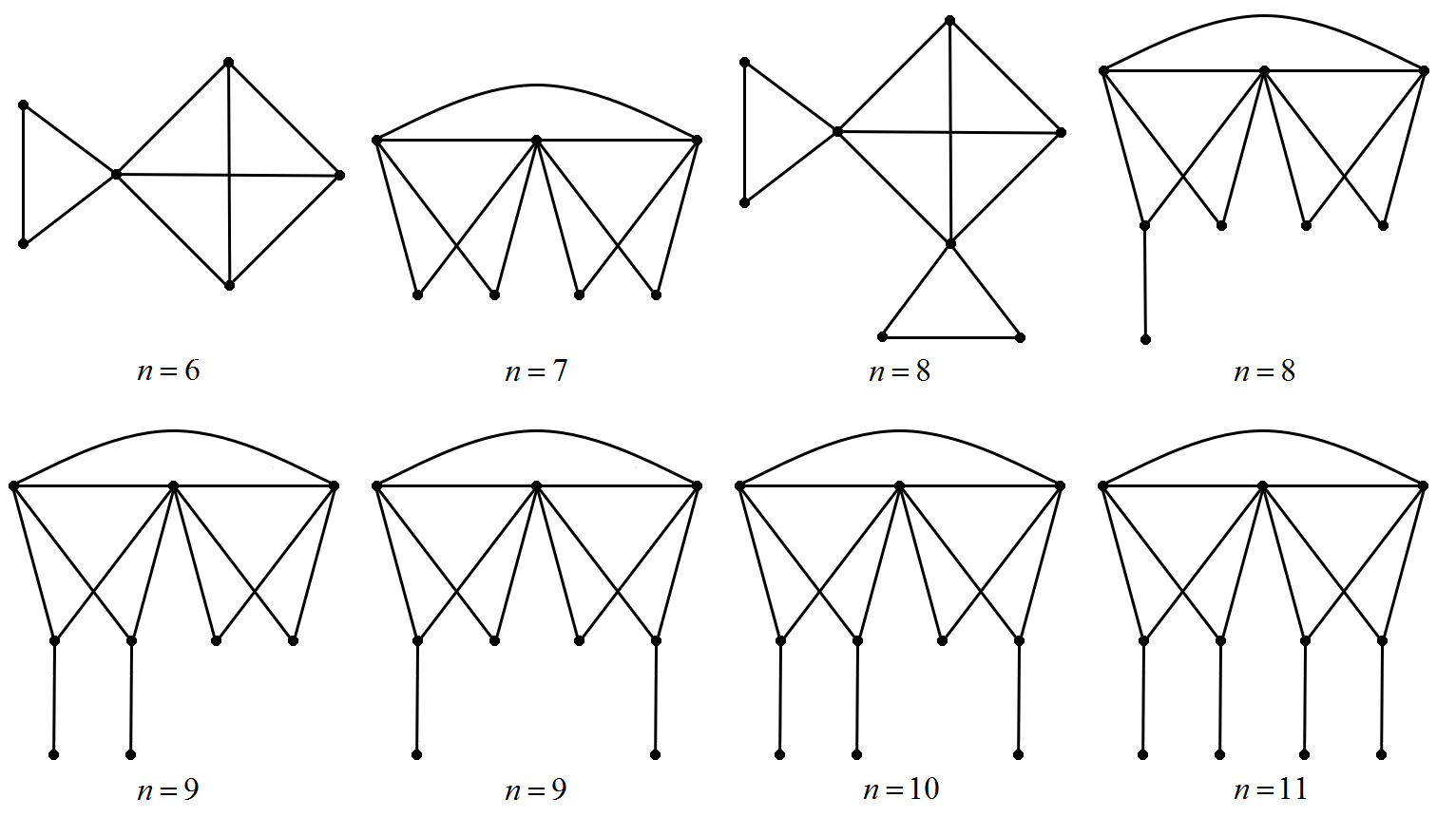}
\caption{All minimum saturated graphs for $n=6,7,8,9,10,11$.}\label{sat6}
\end{figure}
\end{proof}

To complete the proof of Theorem~\ref{THM: 6}, it is sufficient to construct a $\mathcal{C}_{\ge 6}$-saturated graph of order $n\ge 10$ and $\lceil\frac{3}{2}(n-1)\rceil$ edges.
For odd $n\ge 10$, let $M_{6,n}$ be the graph obtained from $D(\frac{n-7}{2},2)$ by pending a leaf to each of its vertex except the center. It is easy to check that $|V(M_{6,n})|=n$ and $|E(M_{6,n})|=\lceil\frac{3}{2}(n-1)\rceil$. By Lemma~\ref{62}, $M_{6,n}$ is a $\mathcal{C}_{\ge 6}$-saturated graph and we are done.
For even $n\ge 10$, let $M_{6,n}$ be obtained by deleting one leaf from $M_{6,n+1}$. Again by Lemma~\ref{62}, $M_{6,n}$ is a $\mathcal{C}_{\ge 6}$-saturated and can be checked that  $E(M_{6,n})=\lceil\frac{3}{2}(n-1)\rceil$.

\section{Structural theorem for $\mathcal{C}_{\ge r}$-saturated graphs and a new lower bound}
For a graph $G$ and  a subset $X\subseteq V(G)$, let $\delta_{G}(X)=\min\{d_{G}(v) : v\in X\}$ and $\Delta_{G}(X)=\max\{d_{G}(v): v\in X\}$. We write $d_{G}(X)=d$ for short if $\delta_{G}(X)=\Delta_{G}(X)=d$. Let $\overline{d}_{G}(X)=\frac{1}{|X|}\sum_{v\in X}d_{G}(v)$ be the average degree of $X$. Let $N_G(X)$ be the set of neighbors of $X$ out of $X$.
For a graph $G$ and two disjoint vertex sets $U, W\subset V(G)$, let $G[U]$ be the subgraph induced by $U$, and $G[U, W]$ be the bipartite subgraph of $G$ with vertex classes $U, W$ and edge set $$E_G[U,W]=\{uv\in E(G):u\in U\mbox{ and }v\in W\}\mbox{.}$$

The following lemma characterised the $C_{\ge 4}$-saturated graphs and will be used in the proofs of this section.
\begin{lem}[Proposition 2.12 in~\cite{Subdivision12}]\label{LEM: 4str}
A graph $G$ is $\mathcal{C}_{\ge 4}$-saturated if and only if

(1) $B_2(G)$ forms a matching of $G$;

(2) every component of $G-B_2(G)$ is isomorphic to $K_1$ or $T(t,0,0,0,0)$ for some $t\ge 1$.

\end{lem}
The following lemma gives the structure of a $\mathcal{C}_{\ge r}$-saturated graph for $r\ge 6$.

\begin{figure}[h]
\centering
\includegraphics[width=3in]{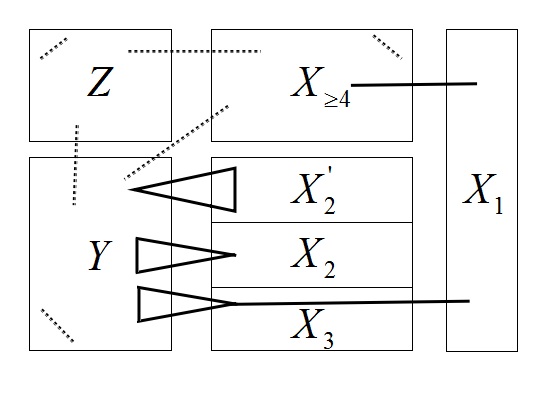}
\caption{The structure of a $\mathcal{C}_{\ge r}$-saturated graph for $r\ge 6$}\label{strc}
\end{figure}

\begin{lem}\label{THM: structure}
Let $G$ be a $\mathcal{C}_{\ge r}$-saturated graph on $n$ vertices for  $n\ge r\ge 6$.
Let $X_1$ be the set of leaves in $G$ and $X_3=\{v\in V(G): d_{G}(v)=3 \mbox{ and } v\in N_{G}(X_1)\}$ and $X_{\ge 4}=\{v\in V(G): d_{G}(v)\ge 4 \mbox{ and } v\in N_{G}(X_1)\}$.
{Let $X_2'$ be the set of vertices of degree two with at least one neighbor of degree two and $X_2$ be the rest of the vertices of degree two.}
Let $Y=N_G(X_2'\cup X_2\cup X_3)\setminus X_1$ and $Z$ be the set of remaining vertices in $G$.
Then  the following hold.

(i) $G[X_1]$, $G[X_1, X_2\cup X_2'\cup Y\cup Z]$, $G[X_2\cup X_3]$, $G[X_2\cup X_3, X_2']$, $G[X'_2\cup X_2\cup X_3, X_{\ge 4}\cup Z]$  are all empty graphs;

(ii) Both $G[X_2']$ and $G[X_1, X_3\cup X_{\ge 4}]$ are perfect matchings;

(iii) For each $uv\in G[X_2']$, there is a $w\in Y$ such that $ w\in N_G(u)\cap N_G(v)$;

(iv) If $Y, Z\cup X_{\ge 4}\neq\emptyset$ then $E_G[Y, Z\cup X_{\ge 4}]\not=\emptyset$;

(v) For each vertex of $X_2\cup X_3$, its two neighbors in $Y$ are adjacent.

(vi) {Let $Y_1$ be the set of isolated vertices} in $G[Y]$ and $Y_2=Y\setminus Y_1$. Let $H=G[Y\cup Z\cup X_{\ge 4}]$. Then $\delta_{H}(Y)\ge 2$ and  $\overline{d}_{H}(Y_2)\ge\frac{5}{2}$.

The structure of $G$ is shown in Figure~\ref{strc}.
\end{lem}
\begin{proof}
(i).
By definition of $X_1$, a component of $G[X_1]$ is either an edge or an isolated vertex.  Since $G$ is connected and $n\ge r\ge 6$, $X_1$ must be an independent set of $G$.

By definition of $Y$ and $Z$, $G[X_1, Y\cup Z]$ is an empty graph. Clearly, every vertex of $X_1$ is contained in a block isomorphic to $K_2$.  If there exists a vertex $v\in N_{G}(X_1)\cap(X_2\cup X_2')$, then $v$ is a cut vertex of $G$ and so $v$ is contained in two adjacent blocks each of which is isomorphic to $K_2$, a contradiction to (b) of Lemma~\ref{LEM:cut}.
Therefore, $G[X_1, X_2\cup X_2']$ is empty.

By definition, $G[X_2]$ is empty.
Suppose there is an edge $uv\in E(G[X_2\cup X_3])$ and $u\in X_3$.  Let $u'$ be the leaf adjacent to $u$.
Since $v\in X_2\cup X_3$ and $G[X_1, X_2]$ is empty, $v$ must have a non-leaf neighbor, say $w$. Then $u'w\notin E(G)$. Thus there is a $(u',w)$-path $P$ of length at least $r-1\ge 5$  in $G$. Clearly, $v\notin V(P)$, otherwise $P=wvuu'$ is of length three, a contradiction. So $P-u'u+uvw$ is a cycle of length at least $r$ in $G$, a contradiction to $G$ is $\mathcal{C}_{\ge r}$-saturated. Therefore, $G[X_2\cup X_3]$ is an empty graph. With a similar discussion, we have that there is no edge $uv$ with $u\in X_3$ (or $u\in X_{\ge 4}$) and $v\in X_2'$ (or $v\in X_2'\cup X_2\cup X_3$). That is $G[X_3, X_2']$ (or $G[X_2'\cup X_2\cup X_3, X_{\ge 4}]$) is empty. Since $E(G[X_2, X_2'])$ is empty by definition, we have $G[X_2\cup X_3, X_2']$ is an empty graph.
By definition, $G[X_2'\cup X_2\cup X_3, Z]$ is empty. So $G[X_2'\cup X_2\cup X_3, X_{\ge 4}\cup Z]$ is empty too. The proof of (i) is complete.
By (i), we have $X_3\cup X_{\ge 4}=N_G(X_1)$ and $X_1,X_2,X_2',X_3,X_{\ge 4},Y,Z$ form a partition of $V(G)$.

(ii). By the definition and (b) Lemma~\ref{LEM:cut}, $G[X_1, X_3\cup X_{\ge 4}]$ is a matching.
To complete (ii), we prove that $\Delta(G[X_2'])=\delta(G[X_2'])=1$.  By definition,  $\delta(G[X_2'])\ge 1$. Suppose there exists a vertex $v\in X_2'$ having two neighbors in $X_2'$, say $u_1,u_2$. Then $u_1u_2\notin E(G)$, otherwise, $G[\{v,u_1,u_2\}]$ forms a component of $G$, a contradiction to the connectivity of $G$. So $G$ contains a $(u_1, u_2)$-path $P$ of length at least $r-1\ge 5$. Clearly, $v\notin V(P)$, otherwise, $P=u_1vu_2$ is of length three, a contradiction. So $P+u_1vu_2$ is a cycle in $G$ of length at least $r+1$, a contradiction.

(iii). Let $uv$ be a component in $G[X_2']$ and $u'$ (resp. $v'$)  be the second neighbor of $u$ (resp. $v$). Then $u',v'\in Y$. To complete (iii), we show  that $u'=v'$. If not, then $u'v\notin E(G)$. Hence $G$ contains a $(u',v)$-path of length at least $r-1\ge 5$ in $G$. With a similar discussion as in (ii), we have $u\notin V(P)$ and so $P+u'uv$ is a cycle of length at least $r+1$ in $G$, a contradiction.

(iv). If not, then $G[Z\cup X_{\ge 4}\cup N_{G}(X_{\ge 4})]$ forms a component of $G$, which is a contradiction to Fact~\ref{FACT: sm}.

(v). If not, then there is a $w\in X_2\cup X_3$ with $N_G(w)\cap Y=\{u,v\}$ but $uv\notin E(G)$. So there is a $(u,v)$-path $P$ of length at least $r-1$ in $G$. With the same reason as in (ii), $w\notin V(P)$. Therefore, $P+uwv$ is a cycle of length at least $r+1$ in $G$, a contradiction.

(vi).
Recall that $H=G[Y\cup Z\cup X_{\ge 4}]$.
We first prove $\delta_H(Y)\ge 2$. Suppose there exists a vertex $v\in Y$ with $d_{H}(v)\le 1$.

If $d_H(v)=0$, then by (v), $E_G[v, X_2\cup X_3]=\emptyset$. So $N_G(v)\subseteq X_2'$. By (iii), the component containing $v$ is isomorphic to $T(t,0,0,0,0)$ for some $t>0$.
By Fact~\ref{FACT: sm}, $G$ is connected, which implies that $G$ is isomorphic to $T(t,0,0,0,0)$. Clearly, $T(t,0,0,0,0)$ is not $\mathcal{C}_{\ge r}$-saturated for $r\ge 6$, a contradiction.

Now suppose $d_H(v)=1$ and let $N_H(v)=\{u\}$. By (v), $N_G(v)\cap(X_2\cup X_3)\subseteq N_G(u)\cap(X_2\cup X_3)$. By (iii), $N_G(v)\cap X_2'$ is disjoint with $N_G(u)\cap X_2'$. We first claim that $N_G(v)\cap X_2'=\emptyset$. If not, choose $w\in N_G(v)\cap X_2'$. Then $wu\notin E(G)$ because $u$ and $v$ have no common neighbor in $X_2'$. So there is a $(u,w)$-path of length at least $r-1\ge 5$  in $G$. Since the  edge containing $w$ in $G[X_2']$ only connect to $v$ in $G$, any $(u,w)$-path must pass through $v$. But the longest $(u,v)$-path in $G$ has length at most two (equality holds when $N_G(v)\cap(X_2\cup X_3)\not=\emptyset$) and the longest $(v,w)$-path has length two, so the longest $(u,w)$-path has length at most four, a contradiction.
With similar discussion, we have $N_G(u)\cap X_2'=\emptyset$.
Therefore, the block $B$ containing $v$ is isomorphic to $H(k, 4, 2)$ centered at $\{u,v\}$, where $k=|N_G(v)\cap(X_2\cup X_3)|+2$. If $|N_G(v)\cap(X_2\cup X_3)|\ge 2$ then $B$ is not $\mathcal{C}_{\ge r}$-saturated because adding any edge in $X_2\cup X_3$ gives rise to a longest cycle of length at most $5\le r-1$ in $B$, a contradiction to the $\mathcal{C}_{\ge r}$-saturation of $B$. So $|N_G(v)\cap(X_2\cup X_3)|\le 1$ and thus $d_G(v)\le 2$, which is a contradiction to $d_G(v)\ge 3$. Therefore, we have $\delta_H(Y)\ge 2$.
Now we show that $\overline{d}_{H}(Y_2)\ge\frac{5}{2}$ using a discharging argument. Recall that every vertex of $Y_2$ has at least one neighbor in $Y_2$.
\begin{claim}\label{CLAIM: c1}
For any $v\in Y_2$ with $d_{H}(v)=2$, the two neighbors of $v$ are adjacent.
\end{claim}
If not, denote $N_H(v)=\{v_1, v_2\}$, then there is a $(v_1,v_2)$-path $P$ of length $r-1\ge 5$ in $G$.
If $v\in V(P)$, by (i), $G[X'_2\cup X_2\cup X_3, X_{\ge 4}\cup Z]$ is empty, then the only vertices used by $P$ are $v_1,v_2,v$ and at most two vertices in $X_2\cup X_3$, i.e. $P$ has length at most $4<r-1$, a contradiction.
Hence, $v\notin V(P)$. It follows that $P+v_1vv_2$ is a cycle of length at least $r+1$ in $G$, a contradiction too.

\begin{claim}\label{CLAIM: c2}
For any pair of vertices $u,v\in Y_2$ with $d_{H}(u)=d_{H}(v)=2$, $uv\notin E(G)$.
\end{claim}
Otherwise, let $w$ be the other neighbor of $v$ in $H$, then $wu\in E(G)$ by Claim~\ref{CLAIM: c1}. Hence the triangle $T=uvwu$ forms a block of $H$. Let $B$ be the block of $G$ containing $T$. Then $B$ is obtained from $T$ by adding $T$-paths of length exactly two, each of which has ends in $\{u,v,w\}$ and the internal vertex in $X_2\cup X_3$.
If $B\cap (X_2\cup X_3)\not=\emptyset$, we claim that $B$ is not $\mathcal{C}_{\ge r}$-saturated, so we have a contradiction to Lemma~\ref{LEM:cut}. In fact, let $P=uxv$ be a $T$-path with $x\in X_2\cup X_3$.
Then $wx\notin E(G)$. But the longest $(w,x)$-path is at most four by the structure of $B$. So $B$ is not $\mathcal{C}_{\ge r}$-saturated.
Now assume $B\cap (X_2\cup X_3)=\emptyset$. That is $B=T=uvwu$. Since  $d_G(u)\ge 3$, by (ii), there must be an edge $u_1u_2\in G[X_2']$ such that the triangle $T'=uu_1u_2u$ forms a block of $G$. Clearly, the longest path connecting any pair of nonadjacent vertices in $V(T)\cup V(T')$ has length at most $4<r-1$, a contradiction.

A vertex $v\in Y_2$ with $d_H(v)=r$ (or $d_H(v)\ge r$) is called an $r$-vertex (or an $r^+$-vertex).
From Claims~\ref{CLAIM: c1} and~\ref{CLAIM: c2}, we have that for each 2-vertex $v\in Y_2$,   either $v$ has two adjacent $3^+$-neighbors in $Y_2$ (we call $v$ an inner vertex), or $v$ has two adjacent neighbors such that one is a $3^+$-vertex in $Y_2$ and the other in $Z\cup X_{\ge 4}$ (we call $v$ a boundary vertex).

\begin{claim}\label{CLAIM: c3}
Every $3^+$-vertex $v\in Y_2$ has at most $d_H(v)-1$ neighbors of degree two in $Y_2$.
\end{claim}
Suppose $v\in Y_2$ is a $3^+$-vertex adjacent to $r$ vertices of degree two in $Y_2$. Let $v_1, \ldots, v_r$ be the 2-vertices in $Y_2$ adjacent to $v$ and $u_1,\ldots, u_r$ be their other neighbors so that $u_i$ is adjacent to $v_i$ for $i=1,\ldots, r$. By Claims~\ref{CLAIM: c1} and~\ref{CLAIM: c2}, $u_1, \ldots, u_r\in N_H(v)$ and $\{v_1, \ldots, v_r\}$ is an independent set in $H$. Hence $d_H(v)\ge r+1$, the equality holds if and only if $u_1=\cdots=u_r$.

\begin{claim}\label{CLAIM: c4}
No $3$-vertex in $Y_2$ is adjacent to two boundary vertices in $Y_2$.
\end{claim}
If not, suppose that there is a $3$-vertex $v\in Y_2$ adjacent to two boundary vertices $v_1, v_2\in Y_2$. By Claim~\ref{CLAIM: c1}, $v, v_1, v_2$ have a common neighbor $u\in Z\cup X_{\ge 4}$. Hence $u$ is a cut vertex separating $v, v_1, v_2$ and the other vertices of $Z\cup X_{\ge 4}$ (if $Z\cup X_{\ge 4}\not=\emptyset$). By definition, $v_1, v_2$ are 2-vertices. By Claim~\ref{CLAIM: c2}, $v_1v_2\notin E(G)$. Hence there is a $(v_1,v_2)$-path $P$ of length at least $r-1$  in $G$. Let $B$ be the block of $G$ containing $\{v, v_1, v_2, u\}$. By (i) and (v), $v\in V(P)$ and the length of $P$ is at most $4<r-1$, a contradiction.

For each $v\in Y_2$, define its initial charge as $ch(v)=d_H(v)-\frac 52$. Then
\begin{equation*}\label{EQN: e1}
\sum_{v\in Y_2}ch(v)=\sum_{v\in Y_2}d_H(v)-\frac 52|Y_2|.
\end{equation*}
Hence to show $\overline{d}_H(Y_2)\ge \frac 52$, it is sufficient to show $\sum_{v\in Y_2}ch(v)\ge 0$.
Now we redistribute the charges according to the following rules.

(R1) Every $3^+$-vertex $v\in Y_2$ gives $\frac 14$ to each of its incident inner vertex in $Y_2$.

(R2) Every $3^+$-vertex $v\in Y_2$ gives $\frac 12$ to each of its incident boundary vertex in $Y_2$.

We proceed to derive  that each vertex $v\in Y_2$ ends up with a nonnegative final charge $ch'(v)$.

For a 2-vertex $v\in Y_2$, if $v$ is an inner vertex, by Claim~\ref{CLAIM: c2}, $v$ has two $3^+$-neighbors in $Y_2$. Hence by (R1), $v$ receives at least $2\times \frac 14=\frac 12$ from its $3^+$-neighbors. If $v$ is a boundary vertex, by (R2), $v$ receives at least $\frac 12$ from its $3^+$-neighbor. So the final charge $ch'(v)=2-\frac 52+\frac 12=0$.

For a 3-vertex $v\in Y_2$,  by Claim~\ref{CLAIM: c4}, if $v$ is adjacent to a boundary vertex then $v$ has no other neighbor of degree two, so $v$ gives $\frac 12$ to its boundary neighbor. If $v$ is not adjacent to boundary vertex then, by Claim~\ref{CLAIM: c3}, $v$ has at most two neighbors of degree two, so $v$ gives at most $2\times \frac 14$ to its neighbors. Therefore, the final charge $ch'(v)=3-\frac 52-\frac 12=0$.

For a $4^+$-vertex $v\in Y_2$, by Claim~\ref{CLAIM: c3}, $v$ has at most $d_H(v)-1$ neighbors of degree two. By (R1) and (R2), $v$ gives at most $\frac 12(d_H(v)-1)$ to its neighbors of degree two. So the final charge $ch'(v)=d_H(v)-\frac 52-\frac 12(d_H(v)-1)=\frac 12d_H(v)-2\ge 0$.

Therefore,
\begin{equation*}\label{EQN: e2}
\sum_{v\in Y_2}ch(v)=\sum_{v\in Y_2}ch'(v)\ge 0.
\end{equation*}
This completes the proof of (vi).
\end{proof}

\begin{cor}\label{COR:ineq}
Let $G$ be a $\mathcal{C}_{\ge r}$-saturated graph on $n$ vertices for  $n\ge r\ge 6$. $X_1$, $X_2$, $X_2'$, $X_3$,  $X_{\ge4}$, $Y_1$, $Y_2$, $Z$ are defined the same as in Lemma~\ref{THM: structure} and let $x_1=|X_1|, x_2=|X_2|, x_2'=|X_2'|, x_3=|X_3|, x_4=|X_{\ge 4}|, y=|Y|, z=|Z|$ and $y_1=|Y_1|$. We have

(a) $x_1=x_3+x_4$ and $n=x_2+x_2'+2x_3+2x_4+y+z$;

(b) $y_1\le \frac{1}{2}x_2'$ and $y\le 2x_2+2x_3+\frac{1}{2}x_2'$;

(c) if $x_2+x_3=0$ and $G[Y\cup Z\cup X_{\ge 4}]$ is a complete graph, then $z+x_4+y=r-1$; otherwise, $x_4+x_3+x_2'\le n-r$ and $3x_2+2x_3+z-\frac{1}{2}x_2'\ge 2r-n$.
\end{cor}
\begin{proof}
(a) follows directly from (ii) of Lemma~\ref{THM: structure}.

(b) By (v) of Lemma~\ref{THM: structure}, $N_G(X_2\cup X_3)\cap Y\subseteq Y_2$. Hence $N_G(Y_1)\subseteq X_2'$. By (ii), (iii) of Lemma~\ref{THM: structure} and the double-counting method, $2y_1=2|Y_1|\le |E_G(Y_1, X_2')|\le |X_2'|=x_2'$. Similarly, we have  $y-y_1=|Y_2|\le |E_G(X_2\cup X_3, Y_2)|=2|X_2\cup X_3|=2x_2+2x_3$. So $y\le 2x_2+2x_3+\frac 12 x_2'$.


(c) If $x_2+x_3=0$ and $G[Y\cup Z\cup X_{\ge 4}]$ is a complete graph, then $G$ is obtained from the complete graph $K_{y+z+x_4}$ by attaching leaves to $X_{\ge 4}$   and $K_3$'s to $Y$. It is easy to check that this graph $G$ is $C_{y+z+x_4+1}$-saturated, which implies $y+z+x_4=r-1$.

 If $x_2+x_3=0$ but $G[Y\cup Z\cup X_{\ge 4}]$ is not a complete graph, then any pair of nonadjacent vertices in $Y\cup Z\cup X_{\ge 4}$ are connected by a path of length at least $r-1$ in $G$. Obviously, all of the vertices in this path are in $Y\cup Z\cup X_{\ge 4}$, which implies $y+z+x_4\ge r$. Note that $n=z+y+x_2'+2x_4$. So $x_3+x_2'+x_4=x_2'+x_4\le n-r$.

Now suppose $x_2+x_3\neq 0$. Denote $H=G[Y\cup Z\cup X_{\ge 4}\cup X_2\cup X_3]$. Since every vertex in $X_2\cup X_3$ has degree exactly two in $H$, $H$ is not a complete graph if $y+z+x_4+x_2+x_3\ge 4$.  If $y+z+x_4+x_2+x_3\le 3$, since each vertex in $X_2\cup X_3$ has two neighbors in $Y$, $y\ge 2$ and thus $y=2$, $x_2+x_3=1$, and $z+x_4=0$. Therefore, $G$ is isomorphic to $T(t,0,0,0,0)$ for some $t\ge 1$ (for $x_2=1$) or is the graph obtained from $T(t,0,0,0,0)$  by attaching one leaf to the vertex in $X_3$ (for $x_3=1$). By Lemma~\ref{LEM: 4str}, $G$ is $C_{\ge 4}$-saturated but not $C_{\ge 6}$-saturated,  a contradiction to $r\ge 6$. Hence $G[Y\cup Z\cup X_{\ge 4}\cup X_2\cup X_3]$ is not a complete graph.
So any pair of nonadjacent vertices is connected by a path $P$ of length at least $r-1$ in $G$. By (i) and (iii) of Theorem~\ref{THM: structure}, $V(P)\subseteq Y\cup Z\cup X_{\ge 4}\cup X_2\cup X_3$. Therefore, $y+z+x_4+x_2+x_3\ge r$. Note that $n=(y+z+x_4+x_2+x_3)+x_4+x_3+x_2'$. So $x_4+x_3+x_2'\le n-r$.
By (b), we have $3x_2+2x_3+z-\frac{1}{2}x_2'\ge y+z+x_2-x_2'=n-2(x_4+x_3+x_2')\ge 2r-n$.
\end{proof}
\begin{cor}\label{COR: (2)}
Let $G$ be a $\mathcal{C}_{\ge r}$-saturated graph on $n$ vertices for some $r\ge 6$ and $\frac n2\le r\le n$. Then $e(G)\ge n+\frac{r}{2}$. 
\end{cor}
\begin{proof}
Let $X_1, X_2, X_2', X_3, X_{\ge4}, Y_1, Y_2, Z$ are defined the same as in Lemma~\ref{THM: structure} and let $x_1, x_2, x_2', x_3, x_4, y, z$ and $y_1$ defined as in Corollary~\ref{COR:ineq}.
Denote $H=G[Y\cup Z\cup X_{\ge 4}]$. Then
\begin{eqnarray}
e(G)&=& e(H)+e_G(X_3\cup X_{\ge 4}, X_1)+e_G(Y, X_2\cup X_3)+e(G[X_2'])+e_G(Y, X_2')\nonumber\\
&=& e(H)+(x_3+x_4)+2(x_2+x_3)+\frac{3}{2}x_2'.\label{EQN: e(G)}
\end{eqnarray}

If $x_2+x_3=0$ and $G[Y\cup Z\cup X_{\ge 4}]$ is a complete graph, then $y+z+x_4=r-1\ge 5$ by (c) of Corollary~\ref{COR:ineq}.
By Equality~(\ref{EQN: e(G)}),
\begin{eqnarray*}
e(G) &=& |E(G[Y\cup Z\cup X_{\ge 4}])|+(x_3+x_4)+\frac{3}{2}x_2'\\
     &=&\binom{r-1}{2}+x_4+\frac{3}{2}x_2'\\
     &=&\binom{r-1}{2}+\frac{1}{2}x_2'+n-(r-1)\\
     &= & n+\frac{1}{2}(r^2-5r+4)+\frac{1}{2}x_2'\\
     &\ge& n+\frac{r}{2}+\frac{1}{2}x_2'\\
    &\ge&  n+\frac r2,
\end{eqnarray*}
where the third equality holds since $n=(z+x_4+y)+x_4+x_2'=r-1+x_4+x_2'$ and the fifth inequality holds since $r\ge 6$.


Now suppose $x_2+x_3\neq 0$ or $G[Y\cup Z\cup X_{\ge 4}]$ is not a complete graph. Then $x_4+x_3+x_2'\le n-r$ and $3x_2+2x_3+z-\frac{1}{2}x_2'\ge 2r-n$ by (c) of Corollary~\ref{COR:ineq}. Let $A=y-(2x_2+2x_3+\frac{1}{2}x_2')$, $B=(x_4+x_3+x_2')-(n-r)$ and $C=(2r-n)-(3x_2+2x_3+z-\frac{1}{2}x_2')$. Then $B,C\le 0$. Counting  $e_G(Y, X_2'\cup X_2\cup X_3)$, we have $A\le 0$.   Thus, since $\frac n2\le r\le n$, we get
$$\left(\frac{r}{2n}-\frac{1}{4}\right)A+\left(\frac{r}{n}-\frac{1}{2}\right)B+\left(\frac{1}{2}-\frac{r}{2n}\right)C\le 0\mbox{.}$$
So
\begin{equation*}
\begin{split}
e(G)  &\ge e(G)+\left(\frac{r}{2n}-\frac{1}{4}\right)A+\left(\frac{r}{n}-\frac{1}{2}\right)B+\left(\frac{1}{2}-\frac{r}{2n}\right)C\\
   &= e(G)+\left(\frac{r}{2n}-\frac{1}{4}\right)\left(z+y+x_2'+x_2+2x_3+2x_4\right)+\frac{x_2'}{8}-\frac{3x_2}{4}-\frac{x_3}{2}-\frac{z}{4}\\
   &\ge \frac{5}{4}n+\frac{1}{8}\left(2z+6x_2+x_2'+4x_3\right)+\left(\frac{r}{2n}-\frac{1}{4}\right)n+\frac{1}{8}\left(x_2'-6x_2-4x_3-2z\right)\\
   &\ge n+\frac{r}{2}+\frac{1}{4}x_2'\\
   &\ge n+\frac r2\mbox{,}
\end{split}
\end{equation*}
where the third inequality holds since $e(G)\ge \frac 54n$ by Theorem~\ref{THM: subdivision}.
\end{proof}

\section{Proof of Theorem~\ref{THM: upper}}
In this section, we construct maximally $\mathcal{C}_{\ge r}$-saturated graphs  that achieve the bounds stated in Theorem~\ref{THM: upper}.
Our constructions are based on the constructions of the maximally nonhamiltonian graphs with fewest edges given in~\cite{Clark83,Clark92,LJZY97,Stacho98}.
Bollob\'as~\cite{bollobas78} posed the problem of finding $\sat(n, C_n)$. Bondy~\cite{bondy72} has shown that $\sat(n, C_n)\ge \lceil\frac {3n}2\rceil$ for $n>7$.
In~\cite{Clark83,Clark92,LJZY97}, the authors completely determined that $\sat(n, C_n)=\lceil\frac {3n}2\rceil$ by constructing the maximally nonhamiltonian graphs with fewest edges. These constructions came from appropriate modifications of a family of well-known snarks, Isaacs' flower snarks.
Let $J_k$ be the Isaacs' flower snark on $4k$ vertices with $k=2p+1$ and $p\ge 7$, and for a vertex $v\in V(J_k)$, $J_k(v)$ denotes the graph obtained from $J_k$ by expanding $v$ to a triangle and for an edge $uv\in E(J_k)$, $J_k(uv)$ denotes the graph obtained from $J_k$ by replacing the edge $uv$ by a bowtie (i.e. a $T(2,0,0,0,0)$ in this paper), detailed definitions can be found in~\cite{Stacho98} (Definitions 1, 2 and 3) and {the appendix of this paper.  The following table lists the optimal $C_n$-saturated graphs for all $n$, where Clark et al~\cite{Clark83,Clark92} gave the construction for $n=8p, 8p+2, 8p+4$ and $8p+6$, and the optimality of the other cases have been proved by Stacho~\cite{Stacho98}.

\begin{center}
\begin{tabular}{c|c|c|c}\label{TB: t1}
order& constrction& order& constrction\\
\hline
$8p$ & $J_{k-2}(v_2,v_{14})$ & $8p+1$ & $J_{k-2}(v_{14})(v_0v_2)$\\
$8p+2$ & $J_{k-2}(v_2,v_{14},v_{26})$ & $8p+3$ & $J_{k-2}(v_{14},v_{26})(v_0v_2)$\\
$8p+4$ & $J_k$ & $8p+5$ & $J_{k-2}(v_{14},v_{26},v_{38})(v_0v_2)$\\
$8p+6$ & $J_k(v_2)$ & $8p+7$ & $J_k(v_0v_2)$
\end{tabular}

\end{center}


We define an {\it almost 3-regular} graph is a graph with all vertices of degree three but one, say $u_0$, of degree four with the property that  the neighborhood $N_G(u_0)$ induces a perfect matching in $G$, say $\{u_1u_2, v_1v_2\}$, such that $u_1, u_2$ (resp. $v_1, v_2$) have distinct neighbors out of $\{u_0\}\cup N_G(u_0)$. Note that $G[N_G(u_0)\cup\{u_0\}]\cong T(2,0,0,0,0)$ by the definition. A {\em barbell} is a graph obtained from two disjoint triangles by adding a new edge connecting them.
For simplify, we call  a 3-regular (or an almost 3-regular) graph containing no barbell as a subgraph {\it  a good graph}.
By the definitions of $J_k$, $J_k(v)$ and $J_k(uv)$, we can check that all optimal graphs constructed in the above table are good. So we have
\begin{lem}\label{LEM: 56}
 For any $r\ge 56$, there exists a ${C}_{r}$-saturated good graph $G$ on $r$ vertices and $\lceil\frac{3r}{2}\rceil$ edges.
\end{lem}

We also need the following property of $C_r$-saturated graph on $r$ vertices.
\begin{lem}\label{LEM: mnh}
Let $G$ be a ${C}_{r}$-saturated good graph on $r\ge 6$ vertices. Then every edge $e\in E(G)$ is contained in a cycle of length $r-1$.
\end{lem}
\begin{proof}
Suppose there exists an edge $e=u_0v_0\in E(G)$ which is not contained in any cycle of length $r-1$.

\noindent{\bf Case 1:} $d_G(u_0)=d_G(v_0)=3$.

Let $N_G(u_0)=\{v_0,a_1,a_2\}$ and $N_G(v_0)=\{u_0,b_1,b_2\}$.

Suppose $a_1a_2, b_1b_2\in E(G)$.
If $|\{a_1,a_2\}\cup\{b_1,b_2\}|=2$, then $\{a_1,a_2\}=\{b_1,b_2\}$ and $G[\{v_0,u_0,a_1,a_2\}]$ is isomorphic to $K_4$. Since $G$ is connected and $r\ge 6$, $G$ must be an almost 3-regular graph and the unique 4-vertex is in $\{a_1,a_2\}$. But this is impossible since $G[N_G(a_i)\cup\{a_i\}]\ncong T(2,0,0,0,0)$ for $i=1,2$.
If $|\{a_1,a_2\}\cup\{b_1,b_2\}|=3$, without loss of generality, let $a_1=b_1$ and $a_2\neq b_2$, then $a_1$  is a 4-vertex in $G$ and so $G$ must be an almost 3-regular graph.  Note that $N_G(a_1)=\{a_2,b_2,v_0,u_0\}$. So $G[N_G(a_1)\cup\{a_1\}]\ncong T(2,0,0,0,0)$,
a contradiction.
So $|\{a_1,a_2\}\cup\{b_1,b_2\}|=4$. But $G[\{u_0,a_1,a_2,v_0,b_1,b_2\}]$ induces a barbell in $G$, a contradiction.

Now suppose one of $a_1a_2,b_1b_2$ is not an edge in $G$. Without loss of generality, assume $a_1a_2\notin E(G)$. Then $G$ contains a Hamiltonian $(a_1, a_2)$-path $P$. So $u_0$ is an internal vertex of $P$.
We claim that $u_0v_0\in E(P)$. If not, then $u_0a_1,u_0a_2\in E(P)$ and so $P=a_1u_0a_2$, a contradiction. Thus $u_0v_0\in E(P)$. Since one of $u_0a_1,u_0a_2$ is contained in $P$, without loss of generality, assume $u_0a_1\in E(P)$. Hence $P-u_0a_1$ is a $(u_0, a_2)$-path on vertex set $V(G)\setminus\{a_1\}$.  Since $a_2u_0\notin E(P)$, $P-u_0a_1+u_0a_2$ is a cycle of length $r-1$ containing $u_0v_0$, a contradiction.

\noindent{\bf Case 2:} One of $u_0,v_0$ is a 4-vertex.

 Without loss of generality, assume $d_G(u_0)=4$ and $N_G(u_0)=\{u_1,u_2,v_0,v_1\}$ with $u_1u_2, v_0v_1\in E(G)$. Let $N_{G}(v_0)=\{u_0, v_1, a\}$. By the definition of the almost $3$-regular graph, $av_1\notin E(G)$. Hence $G$ contains a Hamiltonian path connecting $a$ and $v_1$.
 If $u_0v_0\notin E(P)$, then $P=av_0v_1$ is of length $2$, a contradiction. Thus $u_0v_0\in E(P)$. Since there is another one of $v_0v_1, v_0a$ contained in $P$, without loss of generality, assume $v_0v_1\in E(P)$. Then $v_0a\notin E(P)$. Hence $P-v_0v_1+v_0a$ is a cycle on $r-1$ vertices  containing $u_0v_0$ in $G$, a contradiction.
\end{proof}

Let $G$ and $H$ be two distinct graphs and $v\in V(G)$. We {\em attach $H$ to $v$} means that we identify a vertex of $H$ and $v$ to obtain a new graph.
Let $U, W$ be two disjoint subsets of $V(G)$. We define $L(G; U, W)$ be the graph obtained from $G$ by attaching a $K_2$ to each vertex of $U$ and attaching a $K_3$ to each vertex of $W$.
A vertex is called a {\it support vertex} of $G$ if it is adjacent a leaf of $G$.
For two graphs $G, H$, let $u$ and $v$ be two support vertices of $G$ and $H$, respectively. Define $C(G, H; uv)$ to be the graph obtained from $G$ and $H$ by adding a new edge $uv$ and deleting the leaves adjacent to $u, v$ in $G$ and $H$.
For $k\ge 3$ and a sequence of graphs $G_1, G_2, \ldots, G_k$. We recursively define $$C(G_1,...,G_k; u_1v_1,\ldots, u_{k-1}v_{k-1})=C(C(G_1,...,G_{k-1};u_1v_1,\ldots,u_{k-2}v_{k-2}),G_k; u_{k-1}v_{k-1}),$$
where $u_i$ (resp. $v_i$)is a support vertex of $G_i$ (resp. $G_{i+1}$).

Let $M_{r,r}$ be a ${C}_{r}$-saturated good graph. We define $M_{r,n}$ as follows:
\begin{itemize}
\item
If $r\le n\le 2r$, define $M_{r,n}=L(M_{r,r}; U, \emptyset)$, where $U\subset V(M_{r,r})$ and $|U|=n-r$;

\item
if  $2(k-1)r-2(k-2)<n<\frac{4k-3}{2}r$ for some $k\ge 2$, define
 $$G=C(G_1,...,G_{k-1};u_1v_1,\ldots, u_{k-2}v_{k-2}),$$
where $G_i=L(M^i_{r,r}; U_i, V_i)$ and $M^i_{r,r}$ are pairwise disjoint copies of $M_{r,r}$, $U_i(\not=\emptyset)$ and $V_i$ form a partition of $V(M^i_{r,r})$ with $\sum_{i=1}^{k-1}|V_i|=n-2(k-1)r+2(k-2)$, and $v_{i-1}, u_i\in U_i$ and $v_{i-1}\not=u_{i}$ for $1\le i\le k-1$.

\item
if  $\frac{4k-3}{2}r\le n\le 2kr-2(k-1)$ for some $k\ge 2$, define
 $$M_{r,n}=C(G_1,...,G_{k}; u_1v_1,\ldots, u_{k-1}v_{k-1}),$$
where $G_i=L(M^i_{r,r}; U_i, \emptyset)$ and $M^i_{r,r}$ are pairwise disjoint copies of $M_{r,r}$, $U_i\not=\emptyset$ are subsets of $V(M^i_{r,r})$ with $\sum_{i=1}^{k}|U_i|=n-kr+2(k-1)$,  and $v_{i-1}, u_i\in U_i$ and $v_{i-1}\not=u_i$ for $1\le i\le k$.

\end{itemize}



\begin{prop}\label{PROP: p1}
For $n\ge r\ge 56$, $M_{r,n}$ is $\mathcal{C}_{\ge r}$-saturated graph.
\end{prop}
\begin{proof}
It is sufficient to show that  $C(G_1, \ldots, G_k; u_1v_1, \ldots, u_{k-1}v_{k-1})$ is $\mathcal{C}_{\ge r}$-saturated for $k\ge 1$, where $G_i=L(M^i_{r,r}; U_i, V_i)$ and $M^i_{r,r}$ are pairwise disjoint copies of $M_{r,r}$, $U_i(\not=\emptyset)$ and $V_i$ are disjoint subsets of $V(M^i_{r,r})$, and $v_{i-1}, u_i\in U_i$ with $v_{i-1}\not=u_i$ for $1\le i\le k-1$.

Let $H=C(G_1, \ldots, G_k; u_1v_1, \ldots, u_{k-1}v_{k-1})$.
By definition, the blocks of $H$ are isomorphic to $M_{r,r}$, $K_3$, or $K_2$. So $H$ is $\mathcal{C}_{\ge r}$-free since $M_{r,r}$ is ${C}_{r}$-saturated.
Now we prove that for any $a,b\in V(H)$ with $ab\notin E(H)$, $H$ contains an $(a,b)$-path on at least $r$ vertices.

\noindent{\bf Case 1:} $a,b\in V(G_i)$ for some $1\le i\le k$.

Without loss of generality, assume $i=1$.
If $a,b\in V(M^1_{r,r})$, we are done since $M_{r,r}$ is ${C}_{r}$-saturated. If $a\in V(M^1_{r,r})$ but $b$ is not, then $b$ has a neighbor, say $b'$, in $V(M^1_{r,r})$.
If $ab'\in E(M^1_{r,r})$ then $ab'$ is contained in a cycle $C$ on $r-1$ vertices within $M^1_{r,r}$ by Lemma~\ref{LEM: mnh}. Thus, $C-ab'+bb'$ is an $(a,b)$-path  on $r$ vertices in $L(M^1_{r,r}; U_1, V_1)$. If $ab'\notin E(M^1_{r,r})$ then $M^1_{r,r}$ contains a $(a, b')$-path $P$ on $r$ vertices. Thus $P+bb'$ is an $(a,b)$-path  on $r+1$ vertices in $L(M^1_{r,r};U_1, V_1)$. If $a,b\notin V(M^1_{r,r})$ then $a, b$ have two different neighbors in $V(M^1_{r,r})$ (this is because $ab\notin E(G_1)$).   If $a'b'\in E(M^1_{r,r})$ then $M^1_{r,r}$ contains a cycle $C$ on $r-1$ vertices containing $a'b'$. Thus $C-a'b'+a'a+b'b$ is an $(a,b)$-path on $r+1$ vertices in $L(M^1_{r,r}; U_1, V_1)$, we  are done. If $a'b'\notin E(M^1_{r,r})$ then $M^1_{r,r}$ contains an $(a',b')$-path $P$ on $r$ vertices. So $P+a'a+b'b$ is an $(a,b)$-path on $r+2$ vertices in $L(M^1_{r,r}; U_1, V_1)$.

\noindent{\bf Case 2:} $a\in V(G_i)$ and $b\in V(G_{j+1})$ for some $1\le i\le j\le k-1$.

Since $H$ is connected and $u_jv_j$ is a cut edge, there exists an $(a, u_j)$-path  $P_1$ from $a$ to $u_j$ containing no vertices in $V(M^{j+1}_{r,r})$. If $b\neq v_j$, then $u_jb\notin E(H)\cap E(G_{j+1})$. Hence $G_{j+1}$ contains a $(u_j, b)$-path $P_2$ on at least $r$ vertices by Case 1. Hence $P_1+P'_2$ is an $(a, b)$-path on at least $r$ vertices in $H$. If $b=v_j$, we may also assume $a=u_{i}$ by symmetry, which implies $i<j$.
Let $P$ be a $(u_i, v_{j-1})$-path $P_1$ containing no vertices in $V(M^{j}_{r,r})\setminus\{v_{j-1}\}$. Since $v_{j-1}\neq u_j$, $v_{j-1}v_j\notin E(G_{j})$. Again from Case 1, $G_j$ contains a $(v_{j-1}, v_j)$-path $P_2$ on at least $r$ vertices.  Hence $P_1+P_2$ is an $(a,b)$-path on at least $r$ vertices in $H$.
\end{proof}

\begin{prop}\label{PROP: p2}
For $n\ge r\ge 56$,
$M_{r,n}$ is $\mathcal{C}_{\ge r}$-saturated with $e(M_{r,n})=g(\frac{r}{n})n+O(\frac{n}{r})$. Furthermore, if $r\le n\le 2r$, we have $\sat(n, \mathcal{C}_{\ge r})=n+\lceil\frac{r}{2}\rceil$.
\end{prop}
\begin{proof}
By  Theorem~\ref{LEM: 56} and Proposition~\ref{PROP: p1}, we know that $M_{r,n}$ is $\mathcal{C}_{\ge r}$-saturated. In the following, we check the order and the number of edges of $M_{r,n}$.
If $r\le n\le 2r$, $|V(M_{r,n})|=r+|U|=n$ and $e(M_{r,n})=\lceil\frac{3r}2\rceil+n-r=n+\lceil\frac{r}2\rceil$.
Since $\sat(n, \mathcal{C}_{\ge r})\ge n+\frac{r}{2}$, we have $\sat(n, \mathcal{C}_{\ge r})=n+\lceil\frac{r}{2}\rceil$.
If  $2(k-1)r-2(k-2)<n<\frac{4k-3}{2}r$ for some $k\ge 2$, by definition,
$$|V(M_{r,n})|=(k-1)r+\sum_{i=1}^{k-1}|U_i|+2\sum_{i=1}^{k-1}|V_i|-2(k-2)=n$$
and
\begin{eqnarray*}
e(M_{r,n})&=&(k-1)e(M_{r,r})+\sum_{i=1}^{k-1}|U_i|+3\sum_{i=1}^{k-1}|V_i|-(k-2)\\
    &=&(k-1)\left\lceil\frac{3}{2}r\right\rceil+(k-1)r+2\left(n-2(k-1)r+2(k-2)\right)-(k-2)\\
    &=&2n-(k-1)\left\lfloor\frac{3}{2}r\right\rfloor+3(k-2)\\
    &=& g\left(\frac{r}{n}\right)n+O\left(\frac{n}{r}\right)\left(<g\left(\frac{r}{n}\right)n+\frac{2n}r\right).
\end{eqnarray*}
If  $\frac{4k-3}{2}r\le n\le 2kr-2(k-1)$ for some $k\ge 2$, by definition,
$$|V(M_{r,n})|=kr+\sum_{i=1}^{k}|U_i|-2(k-1)=n$$
and
\begin{eqnarray*}
e(M_{r,n})&=&k e(M_{r,r})+\sum_{i=1}^{k}|U_i|-(k-1)\\
    &=&k\left\lceil\frac{3}{2}r\right\rceil+n-kr+2(k-1)-(k-1)\\
    &=&n+k\left\lceil\frac{r}{2}\right\rceil+(k-1)\\
    &=& g\left(\frac{r}{n}\right)n+O\left(\frac{n}{r}\right)\left(<g\left(\frac{r}{n}\right)n+\frac{2n}r\right).
\end{eqnarray*}
\end{proof}

\section{Remarks}
It is obvious that the Tur\'an function has monotonicity, i.e., $\ex(n,\mathcal{F}_1)\ge \ex(n,\mathcal{F}_2)$ for  $\mathcal{F}_1\subseteq\mathcal{F}_2$. But the saturation number does not have this  property (as has been observed in~\cite{Subdivision12,FF11,KT86,Pi04}  ).
In this paper,
we determine the exact values of $\sat(n,\mathcal{C}_{\ge r})$ for $r=6$ and $\frac n2\le r\le n$.
From the image of $g(x)$, we guess that $\sat(n,\mathcal{C}_{\ge r})$ does not have monotonicity with respect to $r$ too.
It is also an interesting question to determine the exact values of $\sat(n,\mathcal{C}_{\ge r})$ for the other values of $r$. It seems like for $r\ge 7$, $\sat(n,\mathcal{C}_{\ge r})$ are always close to either $\frac{5n}{4}$ or $\frac{3n}{2}$ when $n$ is large.

\section{Appendix}
In this section, we give the construction of Isaacs' snarks $J_k$ for odd $k\ge 3$ and two modifications for $3$-regular graphs used in the table of Section 4. We also show the optimality for $n=8p+1, 8p+3, 8p+5, 8p+7$ in that table.
\begin{myDef}
For odd $k\ge 3$, the Issacs' snark $J_k$ is a $3$-regular graph on vertex set $V=\{v_0,v_1,...,v_{4k-1}\}$ with edge set
$$E=\bigcup_{j=0}^{k-1}\{v_{4j}v_{4j+1},v_{4j}v_{4j+2},v_{4j}v_{4j+3},v_{4j+1}v_{4j+7},v_{4j+2}v_{4j+6},v_{4j+3}v_{4j+5}\}$$
in the sense of module $4k$.
\end{myDef}
The first modification is replacing a vertex of degree $3$ with a triangle in a graph.
\begin{myDef}
For a graph $G$ with a vertex $v$ of degree $3$, let $N_{G}(v)=\{x,y,z\}$. Then $G(v)$ is a new graph with vertex set $V(G(v))=(V(G)\backslash \{v\})\cup\{u_1,u_2,u_3\}$ where $u_1,u_2,u_3\notin V(G)$ and edge set$$E(G(v))=(E(G)\backslash\{vx,vy,vz\})\cup\{u_1u_2,u_2u_3,u_3u_1,u_1x,u_2y,u_3z\}\mbox{.}$$
\end{myDef}
We also write $G(u_1)(u_2)...(u_k)$ by $G(u_1,u_2,...,u_k)$ for short.

The second modification is replacing an edge whose endpoints both have degree $3$ with a $T(2,0,0,0,0)$.
\begin{myDef}
For a graph $G$ and an edge $uv\in E(G)$, where $N_{G}(u)=\{v,x_u,y_u\}$ and $N_{G}(v)=\{u,x_v,y_v\}$. Let $F$ be a copy of $T(2,0,0,0,0)$ on vertex set $\{x_1,x_2,y_1,y_2,z\}$ and edge set $\{x_1y_1, x_2y_2, zx_1, zx_2, zy_1, zy_2\}$, where $V(F)\cap V(G)=\emptyset$. Then $G(uv)$ is a new graph with vertex set $V(G(uv))=(V(G)\backslash\{u,v\})\cup V(F)$ and edge set
$$E(G(uv))=(E(G)\backslash\{uv,ux_u,uy_u,vx_v,vy_v\})\cup E(F)\cup\{x_1x_u,y_1y_u,x_2x_v,y_2y_v\}\mbox{.}$$
\end{myDef}
With these definitions, it is easy to see that the constructions in the table of Section 4 are all good graphs with $r$ vertices and $\lceil\frac{3r}{2}\rceil$ edges for some $r\ge 56$. Since the constructions for $n=8p, 8p+2, 8p+4$ and $8p+6$ in the table was given in~\cite{Clark83,Clark92}, to complete the proof of Lemma~\ref{LEM: 56}, we only need to show the optimality of the other cases. Fortunately, this was almostly proved by Stacho~\cite{Stacho98}.
\begin{prop}[Proposition 2 in~\cite{Stacho98}]\label{prop2cite}
For graph $G=J_k(v_{4i_1+2},...,v_{4i_m+2})(v_{4i_{m+1}}v_{4i_{m+1}+2})$, if $k\ge 5$ odd, $m\ge 0$, $0\le i_l\le k-1$ for $l=1,...,m+1$ and the distance (i.e. the minimal number of edges of a path between two vertices in $J_k$) $d_{J_k}(v_{4i_l+2},v_{4i_p+2})\ge 3$ for any $l\neq p$, then $G$ is a $C_r$-saturated graph on $r$ vertices, where $r=4k+2m+3$.
\end{prop}
When $p\ge 7$ and $k=2p+1$, it is not difficult to show that the distance of any pair from $v_2,v_{14},v_{26},v_{38}$ in $J_{k-2}$ is at least $3$. In fact, by induction, we can show that $d_{J_{k-2}}(v_{4a+2},v_{4b+2})=\min\{|b-a|, k-2-|b-a|\}$. Then by Proposition~\ref{prop2cite}, the optimality for $n=8p+1, 8p+3, 8p+5, 8p+7$ in Table~\ref{TB: t1} is proved.

\end{document}